\documentclass[12pt]{amsart}

\usepackage{geometry}
\usepackage{amsmath}
\usepackage{amssymb,amsfonts,latexsym}

\usepackage{verbatim}
\usepackage[all]{xy}

\newtheorem{thm}{Theorem}[section]
\newtheorem{cor}[thm]{Corollary}
\newtheorem{lem}[thm]{Lemma}
\newtheorem{prop}[thm]{Proposition}
\theoremstyle{definition}

\theoremstyle{remark}
\newtheorem{rem}[thm]{Remark}

\newtheorem{claim}[thm]{Claim}
\newtheorem{exam}[thm]{Example}
\numberwithin{equation}{section}

\long\def\forget#1\forgotten{}

\newcommand{\mQ}{\mathbb{Q}}

\newcommand{\mC}{\mathbb{C}}

\newcommand{\mP}{\mathbb{P}}

\newcommand{\Z}{\mathbb{Z}}

\newcommand{\ra}{\to}

\newcommand\suchthat{{\,|\ }}

\DeclareMathOperator\Gal{Gal}
\def\({\left(}
\def\){\right)}

\newcommand\oline[1] {{\overline{#1}}}

\newcommand\Red{{\operatorname{Red}}}
\newcommand\R{{\operatorname{Red}}}

\newcommand\Aut{{\operatorname{Aut}}}
\newcommand\Inn{{\operatorname{Inn}}}
\newcommand\Out{{\operatorname{Out}}}

\newcommand\orb{{\operatorname{orb}}}

\newcommand\core{{\operatorname{core}}}

\DeclareMathOperator\Mon{Mon}
\DeclareMathOperator\PSL{PSL}

\DeclareMathOperator\Sym{Sym}
\DeclareMathOperator\soc{soc}

\DeclareMathOperator\PGammaL{P\Gamma{L}}

\usepackage[colorlinks,pagebackref,pdftex, bookmarks=false]{hyperref}

\newcommand\DN[1] {{\color{black} {#1}}}
\newcommand\JKo[1] {{\color{black} {#1}}}

\begin{document}

\title[Reducible fibers] 
{Reducible fibers of polynomial maps}%
\begin{abstract}
For a degree $n$ polynomial $f\in \mQ[x]$, the elements in the fiber $f^{-1}(a)\subseteq \mC$ are of degree $n$ over $\mQ$ for most  values $a\in \mQ$ by Hilbert's irreducibility theorem. Determining the set of exceptional $a$'s  without this property is a long standing open problem that is closely related to the Davenport--Lewis--Schinzel problem (1959) on reducibility of separated polynomials.  As opposed to previous work which mostly focuses on indecomposable $f$, we answer both problems for decomposable  $f=f_1\circ \cdots\circ f_r$, as long as the indecomposable factors $f_i\in \mQ[x]$ are of degree $\geq 5$ and are not $x^n$ or a Chebyshev polynomial composed with linear polynomials. 
\end{abstract}
\author{Joachim K\"onig}
\address{Department of Mathematics Education, Korea National University of Education, Cheongju, South Korea}

\def\technion{Department of Mathematics, Technion - Israel Institute of Technology, Israel}
\def\ann{Department of Mathematics, University of Michigan, Ann Arbor}
\author{Danny Neftin}
\address{\technion}

\maketitle
\section{Introduction}\label{sec:intro}
Given a degree $n$ (rational) map $f:X\ra\mP^1_\mQ$ (between smooth projective curves), Hilbert's irreducibility theorem \cite{Hilbert} asserts the existence of infinitely many $a\in \mQ$ for which the fiber $f^{-1}(a)\subseteq \mC$  
 is irreducible\footnote{Equivalently it asserts the existence of infinitely many $a\in \mQ$ such that  $F(a,x)\in \mQ[x]$ is irreducible for a polynomial $F(t,x)\in \mQ(t)[x]$ defining a curve birational to $X$ over $\mQ$. This is also equivalent to $\mP^1_\mQ$ having the Hilbert property, a property of high interest in recent years \DN{\cite{BSFP,CZ}}.} over $\mQ$, that is, it's  elements are of degree $n$ over $\mQ$.
\DN{Moreover}, letting $\Red_f=\Red_f(\mQ)$ denote the set of values $a\in \mQ$ over which the fiber $f^{-1}(a)$ is reducible, Hilbert's theorem asserts that $\Red_f$ is the
 union of finitely many value sets $h_i(Y_i(\mQ))$ of  nontrivial maps $h_i:Y_i\ra\mP^1_\mQ$, $i=1,\ldots,r$.  
\DN{However, determining which of the value sets $h_i(Y_i(\mQ))$ are infinite is far from known.}


\DN{For polynomial maps $f\in \mQ[x]$}, determining the infinite value sets \JKo{contributing to} $\Red_f$ is a long standing open problem of special interest,  firstly due to its relation with reducibility of separated polynomials:  for $h_i\in \mQ[x]$, the value set $h_i(\mQ)$ is contained in $\Red_f$ if and only if $f(x)-h_i(y)\in \mQ[x,y]$ is reducible.
The reducibility of such polynomials plays a key role in studying the rational points on components of the associated curves $f(x)=h(y)$ \cite{AZ2,BT,Pak, DHSWWXZ} for $h\in \mQ[y]$, a problem with a wide range of applications e.g.\ to functional equations \cite{Ritt,Fried73b,MZ}, dynamics \cite{MS,GTZ},  and complex analysis \cite{Pak2}.

Secondly, the problem arises in arithmetic dynamics in the context of stability and newly reducible values, cf.\  \cite[\S 19]{dynamics},\cite{Newly,Newly2}. 
Namely, it is unknown for which integers $m=m_f\geq 2$ there exists (resp.\ exist infinitely many) \DN{$a\in \mQ$} over which the {\it $m$-th iterate} $f^{\circ m}$ of $f$ is {\it newly reducible}, that is, the fiber of $f^{\circ (m-1)}$ over $a$ is irreducible over $\mQ$, but the fiber of $f^{\circ m}$ over $a$ is reducible over $\mQ$.
Finally, the problem also arises in the context of a prime number theorem  in short intervals for function fields \cite{BBSR} as the latter can be viewed as studying the irreducibility of a polynomial upon changing only few of its coefficients, cf.\ \cite{Mul5}. 

Most known results \DN{concern} indecomposable maps $f$, that is, maps that cannot be written as the composition of two maps of degree $>1$. 
For integral values and indecomposable $f\in \mQ[x]$ with\footnote{Moreover, by D\`ebes and Fried \cite{DF}, examples with $\deg f_1 = 5$ actually occur!} $\deg f>5$,    $\Red_f(\Z)$ is the union of a finite set with the integers $f(\mQ)\cap \Z$ in a single value set,  by Fried, cf.~\cite{Mul5}. 
In the more general case of rational values, for indecomposable $f\in \mQ[x]$ with $\deg f>20$, 
$\Red_f(\mQ)$ is the union of the single value set $f(\mQ)$ with a finite set, by M\"uller \cite{Mul} and Guralnick--Shareshian \cite{GS}, cf.\ Corollary \ref{cor:polsQ} and Theorem \ref{thm:pol-cls}. 

This paper deals with decomposable polynomial maps $f=f_1\circ\cdots\circ f_r$ when the monodromy groups of $f_i$, $i=1,\ldots,r$ are nonsolvable. 
This gives the following solution when $f\in \mQ[x]$ has no composition factor of the form $x^n$, or the (normalized) Chebyshev polynomial $T_n$. 
Here,  $T_n\in \mathbb Z[x]$ is the degree $n$ polynomial satisfying $T_n(x+1/x) = x^n + 1/x^n$ for $n\in \mathbb N$. 
\begin{thm}
\label{thm:polsQ}
Suppose $f = f_1\circ \ldots \circ f_r$ for indecomposable $f_i\in \mQ[x]$, $i=1,\ldots,r$ of degree $\geq 5$, none of which equals $\mu_1\circ x^n\circ \mu_2$ or $\mu_1\circ T_n\circ \mu_2$, for $n\in \mathbb N$ and linear $\mu_1,\mu_2\in \mathbb C[x]$. 
 If $\deg f_1>5$, then $\Red_f(\Z)$ is the union of $f_1(\mQ)\cap \Z$ and a finite set. 

 If further  $\deg f_1> 20$, 
then $\Red_f$ is the union of $\R_{f_1}$ and a finite set. In particular, 
 either $\Red_f$ is the union of $f_1(\mQ)$ with a finite set, or 
 $f_1$ is as in Table  \ref{table:two-set-stabilizer}. 
\end{thm}
A surprising part of Theorem \ref{thm:polsQ} is that \DN{if} the fiber of the first polynomial $f_1$ is irreducible over $a$, then  the fibers of the rest of the compositions  $f_1\circ \cdots \circ f_i$  remain irreducible for all but finitely many $a\in\mQ$.
In particular for $m>1$, it follows that the iterate $f^{\circ m}$ is newly reducible only over finitely many values $a\in \mQ$. 
See Theorem \ref{thm:pols} for the analogous result over other fields.  


We further apply our methods to the {\it Davenport--Lewis--Schinzel (DLS)} problem on reducibility of separated polynomials. \DN{This} problem   originates in the late 50's \cite{Sch2,Sch,DLS,DS} in view of the above relation to the curves $f(x)=h(y)$. 
It 
 seeks to determine the polynomials $f,h\in \mC[x]\setminus\mC$ 
for which \DN{$f(x)-h(y)\in\mC[x,y]$} is reducible. 
A trivial case in which $f(x)-h(y)$ is reducible is when $f$ and $h$ have a nontrivial common  left composition factor, that is, $f = g\circ f_1, h = g\circ h_1$ for $g,f_1,h_1\in\mC[x]\setminus \mC$ with $\deg g>1$. 
\DN{The problem is to find the nontrivial reducible cases}.

In case at least one of $f$ and $h$ is indecomposable, the problem is solved by Fried \cite{Fri73}, who gives the possible ramification of $f$ and $h$. The polynomials themselves are then determined by Cassou-Nogu\`es--Couveignes \cite{CCN99}. More recent progress is described in \cite[Theorem 3 and \S3]{AZ2}, \cite{Fried5,FG}, and here as well the main difficulty is the remaining case of decomposable polynomials.


Our methods give the following answer to the DLS problem when one avoids composition factors of the form $x^n$ and $T_n$: 
\begin{thm}\label{thm:DLS-pol}
Let $f,h\in \mC[x]$ be nonconstant polynomials. Assume that $f= 
f_1\circ \ldots \circ f_r$ for indecomposable $f_i\in \mC[x]$ of degree $\geq 5$, none of which is $\mu_1\circ x^n\circ \mu_2$ or $\mu_1\circ T_n\circ \mu_2$, for $n\in \mathbb N$ and linear $\mu_1,\mu_2\in \mathbb C[x]$. 
Assume further that $\deg f_1> 31$. 
Then $f(x)-h(y)\in \mC[x,y]$ is reducible 
if and only if $h=f_1\circ h'$ for some $h'\in \mC[x]$. 
\end{thm}
Note that  common composition factors of $f$ and $h$ necessarily factor through $f_1$ since, for $f_i$'s as above, the decomposition $f=f_1\circ\cdots\circ f_r$ is unique up to composition with linear polynomials  by \DN{Ritt's theorem, see Theorem \ref{thm:Ritt}}. 
We note the degree assumptions on $f_1$ can be removed in both of the above theorems at the account of a longer list of exceptions, cf.\ \S\ref{sec:poly}. 
However, different methods are required for both of the above theorems when $f_i$, $i=1,\ldots,r$ are allowed to be the composition of  $x^n$ or $T_n$ with linear polynomials, that is, 
 when $\Mon(f_i)$, $i=1,\ldots,r$ are allowed to be solvable. Moreover, in such cases $\R_f$ may consist of more than one infinite value set even when the decomposition of $f$ is unique, see Example \ref{ex:Cheb}.

The above two problems share a common ground, namely, they require determining the \DN{maps} $h:Y\ra \mP^1_k$, \DN{from $Y$ of genus $\leq 1$, whose fiber product with $f$ is reducible}. 
The key step in doing so is reducing the problem to 
determining the genus $\leq 1$ \DN{maps} whose fiber product with  $f_1$ is reducible, see \S\ref{sec:reduction}. 
This relies on a combination of Ritt's theorem, group theoretic tools, and a new relation between normal subgroups of the {\it monodromy group} of $f$ with decompositions $f=f_1\circ\cdots \circ f_r$, see Lemma \ref{lem:descent}. 


The \DN{main} property of indecomposable polynomials used in the proof is that either their monodromy groups are solvable or they are  nonabelian almost simple by Burnside's theorem on doubly transitive groups. However, our strategy applies more generally to compositions  of indecomposable maps  whose  monodromy groups 
are nonsolvable, \JKo{but have the property that all their proper quotients are solvable}. 
As opposed to M\"uller's finiteness results \cite{Mul3, Mul4} for indecomposable maps,  for such decomposable maps $\R_f$ may contain many infinite value sets. For the sake of simplicity of this paper, this is carried out separately in \cite{KN2} for integral values. 

The classification of finite simple groups (CFSG) is used in the above reduction process  only for basic assertions regarding the outer automorphism group of a simple group such as Schreir's conjecture, and for the ``further" part of Theorem \ref{thm:polsQ}  (in applying Theorem \ref{thm:nonsolv-quot}). 
The final step of classifying subcovers of the Galois closure of $f_1$ is then carried out by applying the primitive monodromy classification theorems (and hence the CFSG) for polynomials by Feit and M\"uller \cite{Mul}, and Guralnick--Shareshian \cite{GS}, see \S\ref{sec:class} and \S\ref{sec:poly}.
In light of the further development of the classification of primitive monodromy groups \cite{NZ1,FGHM, Ad}, the above results 
\DN{would hopefully} extend to rational functions, \DN{and} to general maps under group theoretic restrictions on their monodromy groups, cf.\  \cite[\S 5.2]{KN3}.

We thank \DN{Robert Guralnick} for helpful discussions. The second author was supported by the Israel Science Foundation (grants 577/15 and 353/21), the U.S.-Israel Binational Science Foundation (grant No. 2014173), and is  grateful for the hospitality of the Institute for Advanced Studies.  All computer calculations were carried out using Magma.

\section{Preliminaries}\label{sec:prelim}
\subsection{Coverings}\label{sec:cover}
Let $k$ be a field of characteristic $0$, and $\oline k$ its algebraic closure.
 An (irreducible branched) {\it covering} $f:X\to Y$ (of curves) over $k$ is a morphism of (smooth irreducible projective) curves defined over $k$. Note that as $X$ may {be} geometrically reducible (i.e., reducible over $\oline k$),  the morphism $f \times_k \oline{k}$ obtained by base change from $k$ to $\oline{k}$ may not be a covering over $\oline k$. 
 A covering $h$ is called a {\it subcover} of $f$ if $f=h\circ h'$ for some covering~$h'$. 
A covering $f$ defines a field extension $k(X)/k(Y)$ via the  injection $f^*:k(Y)\to k(X), h\mapsto h\circ f$. 
Two coverings $f_i:X_i\ra Y$, $i=1,2$ over $k$ are called ($k$-)equivalent if there exists an isomorphism
 $\mu:X_1\ra X_2$ (over $k$) such that $f_1\circ \mu = f_2$. 
Note that for two $k$-equivalent coverings, one has $f_1(X_1(k)) = f_2(X_2(k))$ and hence we may consider the value set of a $k$-equivalence class of coverings. 

Recall  that there is a correspondence between equivalence classes of coverings of $\mathbb{P}^1_k$
 and finite field extensions of $k(t)$, up to $k(t)$-isomorphisms, cf. \cite[Section 2.2]{DL}. 
In particular, letting $\tilde f:\tilde X\ra\mP^1_k$ denote the covering corresponding to the Galois closure $\Omega$ of $k(X)/k(t)$, there is a correspondence between equivalence classes of subcovers $h:Y\ra\mP^1_k$ of $\tilde f$ and subgroups $D\leq A:=\Gal(\Omega/k(t))$. 
Namely, to every such subcover the correspondence associates a subgroup $D\leq A$ (unique up to conjugation) such that $h$ is  equivalent to a covering $f_D:\tilde X/D\ra\mP^1_{k}$ whose composition with the natural projection $\tilde X\ra\tilde X/D$ is $\tilde f$. 

By the {\it genus} of $X$, we mean the genus of a geometrically irreducible component of $X$. Note that since $\oline k/k$ is Galois, it is independent of the choice of the component.

\label{sec:ram} The {\it ramification type} of a covering $f:X\to\mP^1_{\oline k}$ at a point $P\in \mP^1_{\oline k}$ is defined to be the multiset of ramification indices $\{e_f(Q/P)\suchthat Q\in f^{-1}(P)\}$, and the ramification type of $f$ is the multiset of all ramification types over all branch points of $f$. The ramification type of a geometrically irreducible covering $f$ over $k$ is  the ramification type of $f\times_k \oline{k}$.

\JKo{
\subsection{Polynomials and Siegel functions}
A {\it polynomial} covering $f:\mP^1\ra\mP^1$  is a covering  which satisfies $f^{-1}(\infty) = \{\infty\}$ over $\oline k$. In particular, on the affine line, it is given by a polynomial. 
For polynomial coverings, Fried and MacRae \cite{FM69} show that
an indecomposable polynomial  over $k$ is indecomposable  over $\oline k$. 
We shall therefore call such a polynomial simply ``indecomposable" without specifying the base field. Two polynomials $f,g\in k[x]$ are called {\it linearly related} (resp. linearly related {\it over $k$}), if there exist linear polynomials $\mu,\nu\in \overline{k}[x]$ (resp., $\in k[x]$) such that $g=\mu\circ f\circ \nu$.

A {\it Siegel function} is a covering $f:X\ra\mP^1_k$ over a number field $k$, for which $f^{-1}(O_k)$ has infinitely many $k$-rational points. 
Due to Siegel's theorem, \DN{for such a function   one has $X\cong \mathbb{P}^1_k$ and  $\infty$ has at most two preimages over $\overline{k}$. 
With a slight abuse of notation,  coverings with the latter property are also called  Siegel functions.}
}

\subsection{Monodromy} 
Let $f:X\ra\mP^1_k$ be a \DN{covering} over $k$.  Letting  $\Omega$ denote the Galois closure of $k(X)/k(t)$, 
the  {\it arithmetic (resp.~geometric) monodromy group} $A=\Mon_k(f)$ (resp.\ $G=\Mon_{\oline{k}}(f)$) of $f$ is the Galois group $\Gal(\Omega/k(t))$ (resp.~$\Gal(\oline k\Omega/\oline k(t))$) equipped with its permutation action on $A/A_1$, where $A_1=\Gal(\Omega/k(X))$. Note that since $\oline k(t)/k(t)$ is  Galois, so is $k'(t)/k(t)$ for $k'  = \oline k\cap \Omega$. Hence $G\lhd A$. Also note that $f_D$ is geometrically irreducible  if and only if
$\Omega^D/k(t)$ is linearly disjoint from $\oline k(t)$, or equivalently if $\Omega^D\cap k'(t) = k(t)$, that is, if  $D\cdot G = A$.



The following theorem  describes the structure of monodromy groups of polynomials. This classical version is essentially due to Burnside and Schur, cf.\ Section \ref{sec:class} for the full classification result of Feit and M\"uller. 
Note that the CFSG is used in the following theorem only to assert that proper quotients of an almost simple group are solvable, an assertion also known as Schreier's conjecture. 
Denote by $\soc(A)$ the socle of $A$, that is, the product of minimal normal subgroups of $A$.
When $\soc(A)$ is abelian, we say $A$ is an affine permutation group.
\begin{thm}\label{Burnside}
Let $f:\mP^1_k\ra\mP^1_k$ be a polynomial covering with monodromy group $A=\Mon_k(f)$. Then $A$ is either solvable or a  $2$-transitive nonabelian almost simple group. In the former case,  \DN{$f$ is linearly related (over $\oline k$) either to $x^n$, or to} a Chebyshev polynomial, or an indecomposable degree $4$ polynomial. In the latter case, $\soc(A)$ is primitive and  $A/\soc(A)$ is solvable. 
\end{thm}
\begin{proof}
Due to results of Burnside, see \cite{Mul6},  and Schur \cite{Schur}, any primitive group containing a full cycle is known to be either solvable or $2$-transitive; moreover, the minimal normal subgroup of a nonaffine $2$-transitive group is known to be simple and primitive due to a theorem of Burnside \cite[Theorem 7.2E]{DM}.
The indecomposable polynomial coverings $f$ with affine monodromy group $G$, and in particular those with solvable monodromy group, were essentially classified by Chisini and Ritt,  see the proof of \cite[Satz 5]{Hup} or \cite{Mul}. Finally if $A$ is almost simple, then $A/\soc(A)$ is solvable by Schreier's conjecture, which follows from the classsifcation of finite simple groups. 
\end{proof}

In case  $f$ is a Siegel function, M\"uller \cite[Theorem 3.3]{Mul2} gives the following description of $\Mon(f)$. This version requires the CFSG only for Schreier's conjecture and the following bound on orders of elements in the outer automorphism group $\Out(S)$ of a nonabelian simple group $S$ other than $A_5$ and $\PSL_2(7)$:
\begin{equation}\label{equ:o-bound}
 o(\Out(S)) < \frac{\#S}{2o(S)^2}, 
 \end{equation}
 where $o(G)$ denotes the maximal order of an element in a group $G$.  
\begin{thm}
\label{thm:monsiegel}
Let $f:\mP^1_k\ra \mP^1_k$ be a indecomposable Siegel function with nonaffine monodromy group $A=\Mon_k(f)$. 
Then  $A$ is either nonabelian almost simple, or $A\leq (\Aut(S)\times \Aut(S))\rtimes C_2$ contains $S^2$ as a unique minimal normal subgroup for a nonabelian  simple group $S$. In particular, the proper quotients of $A$ are solvable. 
\end{thm}
\begin{proof}
We note how the proof of \cite[Theorem 3.3]{Mul2} adjusts to give this version without relying on the CFSG. Assume $A$ is not almost simple. As in addition $A$ is not affine,  the remaining possibilities for $A$ have product action, or regular normal subgroup action, or diagonal action in the terminology of \cite[\S 2]{Mul2}. Without invoking the CFSG,   \cite[\S 3.4.1]{Mul2} shows that, for $A$ of product action, one has $S^2\lhd A\leq (U\times U)\rtimes C_2$ where $U$ is a primitive   group of degree $r$ which contains an $r$-cycle and has socle $S$. Since $A$ is nonaffine, $S$ is nonabelian and Burnside's theorem shows that $U$ is a $2$-transitive almost simple group, so that $S$ is simple and $U\leq \Aut(S)$, yielding the desired conclusion in the product type case.

Without invoking the CFSG,  \cite[\S 3.5]{Mul2} shows that $A$ cannot have a regular normal subgroup:  More precisely, it shows  that in such case $A$ must be a subgroup of $H^m\rtimes S_m$ equipped with a product action for $H$ which is not $2$-transitive, contradicting the above  conclusion in the product type case. Finally, \cite[\S 3.6]{Mul2} shows that $A$ cannot have a diagonal action using the CFSG only when applying \eqref{equ:o-bound}. 

In total it follows that $A$ has a unique minimal normal subgroup $\soc(A)=S^t$ for a simple group $S$ and $t\in\{1,2\}$, and hence the quotient $A/\soc(A)$ is a subgroup of $\Out(S)^t\rtimes S_t$, $t\leq 2$, which is solvable by Schreier's conjecture. 
\end{proof}

For coverings of genus at most  $1$ the following analogous  assertion is shown in Appendix \ref{app:Shih} using the classification of monodromy groups and hence the CFSG. 
\begin{thm}\label{thm:nonsolv-quot}
Suppose $f:X\ra \mP^1_{\oline k}$ is an indecomposable covering of genus $g_X\leq 1$ and nonaffine monodromy group $G$. 
Then $G/\soc(G)$ is solvable. 
\end{thm}

\subsection{Specializations} \label{sec:spec}
\DN{To a covering $f:X\ra\mP^1_k$, one associates an irreducible polynomial $F\in k(t)[x]$ such that the curve $F(t,x)=0$ is birational to $X$. Further, we replace $F$ by a polynomial $F\in k[t,x]$  by multiplying it with an element of $k(t)$. Note that after these operations, the resulting set of values $t_0\in k$ for which $F(t_0,x)$ is reducible differs from $\Red_f(k)$  only in finitely many values. }

We next recover a well known criterion for the reducibility of $F(t_0,x)$. 
Let $\Omega$ be the splitting field of $F$ over $k(t)$, so that $A=\Gal(\Omega/k(t))$ is the arithmetic monodromy group of $f$. 
A well known fact from algebraic number theory \cite[Lemma 2]{KN18},  asserts that for every $t_0\in k$ which is neither a root nor a pole of the discriminant $\delta_F\in k(t)$ of $F$, the splitting field  $\Omega_{t_0}$ of $F(t_0,x)$ is Galois, and its Galois group is identified as a permutation group with a  subgroup $D\leq A$, unique up to conjugation, known as the decomposition group at \DN{$t\mapsto t_0\in k$}. 
Moreover,   $\Omega^D$ has a degree $1$ place $P$ over $t_0$. 
In particular,  $\Omega^D\cap \oline k(t) = k(t)$,
so that 
the corresponding morphism  
$f_D:X_D\ra\mP^1_k$ is  geometrically irreducible, and $DG=A$. 
The  place $P$ corresponds to a $k$-rational point  $P\in X_D(k)$ such that $f_D(P) = t_0$. 
Since  $D$ and $\Gal(\Omega_{t_0}/k)$ are isomorphic as permutation groups,  $F(t_0,x)$ is reducible if and only if $D$ is intransitive.
\DN{In total:
\begin{prop}
\label{prop:spec}
Let $f:X\ra \mP^1_k$ be a covering with arithmetic and geomertric monodromy groups $A$ and $G$, respectively. 
Let $D=D_{t_0}$ be the  decomposition group at $t\mapsto t_0$, and $f_D:X_D\to \mathbb{P}^1_k$ its corresponding covering.
Then: 
\begin{enumerate} 
\item 
$t_0\in f_D(X_D(k))$ and $DG = A$ for all but finitely many $t_0\in k$; 
\item For all but finitely many $t_0\in k$, $t_0\in \Red_f(k)$ if and only if $D$ is intransitive.
\end{enumerate}
\end{prop}
}

Proposition \ref{prop:spec} implies that $\R_f$ is the union of $\bigcup_D  f_D(X_D(k))$ with a finite set, where $D\leq A$ runs over maximal intransitive subgroups with $DG=A$. 
If $X_D(k)$ is infinite and $k$ is a finitely generated field,  Faltings' theorem implies that  $g_{X_D}\leq 1$. 
Similarly if $k$ is a number field with ring of integers $O_k$ and $f_D(X_D(k))\cap O_k$ is infinite, then Siegel's theorem implies that $f_D$ is a Siegel function. \DN{This is the step in which the finite set of exceptions is no longer explicit; Computing it is in general hopeless since the curves $X_D$ are arbitrary.  We therefore have:}
\begin{cor}
\label{cor:kronecker}\label{cor:faltings}
Let $f:X\to\mP^1_k$ be a covering over a finitely generated field $k$ with arithmetic (resp.~geometric) monodromy $A$ (resp.~$G$). 
Then $\R_f$ and $\bigcup_D f_D(X_D(k))$ differ by a finite set, where $D$ runs over maximal intransitive subgroups of $A$ with $g_{X_D}\leq 1$ and $DG =A$. 

Similarly, if $k$ is a number field and $O_k$ is its ring of integers, then $ \R_f(O_k)$  and $\bigcup_D \left(f_D(X_D(k))\cap O_k\right)$ differ by a finite set, where $D$ runs over maximal intransitive subgroups of $A$ such that $DG=A$ and $f_D$ is a Siegel function. 
\end{cor}

\begin{exam}\label{ex:Cheb}
Let $k:=\mQ(e^{2\pi i/8})$, and 
$f(x):=T_4(x) \in k[x]$. We will show that (1) $\Red_f$ is the union of $f_1(\mQ)\cup h(\mQ)$ with a finite set,
where $f_1(x)=T_2(x)$ and $h(x)=-T_4(x)$. 
Furthermore,  (2) $f_1$ is the the unique indecomposable subcover of the natural projection $f:\mP^1_k\ra \mP^1_k$, $x\mapsto T_4(x)$. 
Since $f_1$ is of degree $2$, it is Galois, and hence $h$ does not factor through $f_1$. 
As pointed out in Section \ref{sec:intro}, this shows that the nonsolvability assumption in Theorem \ref{thm:polsQ} is necessary. 

To show (1) and (2), first note that the Galois closure  of $f$ is the covering $\tilde f:\tilde X\ra\mP^1_k$ by $\tilde X\cong \mP^1_k$ given by  $x\mapsto (x+1/x)^4$, so that $\tilde f = f\circ (x+1/x)$. The arithmetic and geometric monodromy groups $A$ and $G$ of $f$ are the dihedral group $D_4$ of order $8$ equipped with its standard degree $4$ action.
Let $s$ be the automorphism  of $\tilde X$ given by $x\mapsto 1/x$, so that $f$ is equivalent to the subcover $f_{s}: \tilde X/\langle s\rangle\ra\mP^1_k$. 
We next deduce (1) and (2) from:
\begin{claim}\label{claim:sr} $h$ is equivalent to the covering $f_{\langle sr\rangle}:\tilde X/\langle sr\rangle \ra\mP^1_k$. 
\end{claim}
By Corollary \ref{cor:faltings}, it suffices to find the maximal intransitive subgroups $D\leq A$ for which $g_{X_D} \leq 1$ and $DG = A$. 
However, since $\tilde X$ is of genus $0$ and $G=A$, the last two conditions are immediate. 
Up to conjugacy the maximal intransitive subgroups of $D_4$ are $\langle sr\rangle$, and $\langle s,r^2\rangle$. Since $U:=\langle s,r^2\rangle$ is the only intermediate subgroup $\langle s\rangle \leq U\leq D_4$, we deduce that $f_U$ is equivalent to $f_1$, showing that  (1) follows from Claim \ref{claim:sr}.  Since the only proper subgroup of $D_4$ which contains $\langle sr\rangle$ is $\langle sr,r^2\rangle$ and it is not conjugate to $U$, (2) follows. 

It remains to prove Claim \ref{claim:sr}. Note that the composition $\hat f:=T_2\circ \tilde f:\tilde X\ra\mP^1_k$ is a Galois covering with arithmetic monodromy group $D_8$ containing  $A=D_4$ as a subgroup. Since $sr\in A$ and $s\in A$ are conjugate in $D_8$, the coverings $T_2\circ f_{\langle sr\rangle}$ and $T_2\circ f_{\langle s\rangle}$ are equivalent in $D_8$. 
However, since $\langle sr\rangle$ and $\langle s\rangle$ are not conjugate in $A$, $f_{\langle sr\rangle}$ is not equivalent to $f_{\langle s\rangle}$. 
As two covering which are inequivalent but whose compositions with $T_2$ are equivalent,  $f_{\langle sr\rangle}$ is equivalent to $-f_{\langle s\rangle}$: 
$x\mapsto -T_4(x)$.
\end{exam}
For an example over $\mQ$, see \cite[\S2]{Fried100},\cite[Chp.~13, Ex.~1]{FJ}.

\subsection{Decomposable coverings}\label{sec:wreath} Let  $f:Y\to X$ and $h:X\to\mP^1$ be two coverings over $k$ of degrees $m,n$ and monodromy groups \DN{$U\leq \Sym(I),V\leq \Sym(J)$} with point stabilizers $U_1,V_1$, respectively. Then the monodromy group $A$ of $h\circ f$ is naturally a subgroup of the wreath product $U\wr_J V := U^J\rtimes V$, where the semidirect product action of $V\leq \Sym(J)$ permutes the $J$-copies of $U$.  
The action of $A$ is the natural imprimitive degree $mn$ action of $U\wr_J V$ \DN{on $I\times J$ with blocks indexed by $J$}. 

We note two further properties of such monodromy groups $A\leq U\wr_J V$. 
Letting $\Omega_X$ denote the Galois closure of $k(X)/k(\mP^1)$,  the restriction map surjects onto $V=\Gal(\Omega_X/k(\mP^1))$, that is, (1) the projection modulo $U^J$  maps $A$ onto $V$.  
Letting $\Omega$ denote the Galois closure of $k(Y)/k(\mP^1)$ and   $A_0:=A\cap (U^{J}\rtimes \Sym(J\setminus\{0\})$ be the stabilizer of a block $0\in J$,   (2) $A_0$ maps onto $U$ under the projection to the $0$-th coordinate. 
We shall use the following assertions on  decompositions of polynomials:
\begin{lem}\label{lem:pol-kernel}
Suppose $f = u\circ v$ for polynomial coverings $u,v$ of degrees $m,n$, respectively,  then the kernel of the natural projection $\pi$ from $\Mon_k(f)\leq S_n \wr S_m$  to $\Mon_k(u)\leq S_m$ is nontrivial.
\end{lem}
\begin{proof}
Letting $\tilde u$ be the Galois closure of $u$, 
Abhyankar's lemma implies that $e_{\tilde u}(Q/\infty) = m$ for every $Q\in \tilde u^{-1}(\infty)$. Since $e_f(\infty/\infty) = mn$, it follows that $f$ is not a subcover of the Galois closure $\tilde u$ of $u$, and hence the kernel of $\pi$ is nontrivial.  
\end{proof}
The uniquness of decompositions of a polynomial with nonsolvable composition factors is given by Ritt's theorems \cite{Ritt,MZ}. Due to subsequent work of Fried and MacRae \cite[Theorem 3.5]{FM69}, the linear polynomials in the following theorem may even be assumed to be over $k$.
\begin{thm}\label{thm:Ritt} 
Suppose $f=f_1\circ \cdots\circ f_r$  for indecomposable $f_i\in k[x]$ with nonsolvable monodromy group. Then for every decomposition $f=g_1\circ\cdots\circ g_s$ into indecomposables, one has $s=r$ and $g_i = \mu_i\circ f_i\circ\mu_{i-1}$ for linear polynomials $\mu_i\in k[x]$, $i=1,\ldots,r$ with $\mu_0 = \mu_r = id$. 
\end{thm}

\subsection{Fiber products and pullbacks}\label{sec:fiber}
Let $\tilde f:\tilde X\ra Y$ be a Galois covering over $k$ with arithmetic monodromy group $A$. 
Let  $A_1,H\leq A$ be subgroups and   
$f_{A_1}:\tilde X/A_1\ra Y$ and $f_H:\tilde X/H\ra Y$ the corresponding coverings, respectively. 
Setting $X:=\tilde X/A_1$ and $Z:=\tilde X/H$, we denote by $X\# Z$ the (normalization of the) {\it fiber} product of $f_{A_1}$ and $f_H$.

\begin{rem}\label{rem:fiber}
The irreducibility of  $X\# Z$ is equivalent to the linear disjointness of the function fields $k(X)$ and $k(Z)$ over $ k(Y)$, 
which in turn is equivalent to the transitivity of $H$ on $A/A_1$, that is,  $HA_1 = A$. 
When these conditions hold, the natural projection $X\# Z\ra Y$ is equivalent to the covering 
$f_{H\cap A_1}:\tilde X/(H\cap A_1)\ra Y$.
\end{rem}

\begin{lem}\label{Fried}
Let $f:X\ra Y$ and $h:Z\ra Y$ be coverings with reducible fiber product. 
Then  $f=f_0\circ f_1$ where $f_0$ is a subcover of the Galois closure $\tilde h$ whose fiber product with $h$ is reducible. 
\end{lem}
\begin{proof}
Let $g:Z\ra Y$ be a common Galois closure for $f$ and $h$, let $A$ be its (arithmetic) monodromy group, and assume $f\sim g_U$, $h\sim g_V$, and $\tilde h\sim g_N$ 
for $U,V,N\leq A$ with $N = \core_A(V)\lhd A$. 
Since the fiber product of $f$ and $h$ is reducible, $UV\neq A$. 
Since $N\lhd A$, the set $UN$ is a group, and as $U\leq UN$, $f$ factors through  $f_0 :=  g_{UN}:Z/(UN)\ra Y$. Since $UN\leq UV<A$, we have $\deg f_0>1$. Since $UN\cdot V = UV<A$, the fiber product of $f_0$ and $h$ is reducible. 
\end{proof} 

The {\it pullback} of $f$ along $h$ is the natural projection $f_h:W\ra Z$ from $W:=X\#Z$.
\DN{\begin{lem}\label{lem:pull-closure}
Let $f:X\ra Y$ be a covering with Galois closure $\tilde f:\tilde X\ra Y$, and $h:Z\ra Y$ a subcover of $\tilde f$ whose fiber product with $f$ is irreducible. Then the Galois closure of the pullback $f_h$ is equivalent to the projection $\tilde X\ra Z$. 
\end{lem}
\begin{proof}
First replace $h$ by an equivalent subcover  $f_H:\tilde X/H\ra Y$ for some $H\leq \Mon_k(f)$. 
Since $X\#_Y Z$ is irreducible, $H$ is transitive and $f_h$ is equivalent to the projection $\tilde X/(H\cap G_1)\ra \tilde X/H$ by Remark \ref{rem:fiber}, where $G_1\leq \Mon_k(f)$ is a point stabilizer. Moreover, the transitivity of $H$ implies that 
$$\bigcap_{x\in H}(H\cap G_1)^x \subseteq \bigcap_{x\in H} H^x =\bigcap_{g\in G}H^g = 1.$$ Thus the action of $H$ on $H/H\cap G_1$ is faithful, so that there are no nontrivial Galois covers between $\tilde X\ra \tilde X/H$ and $f_h:\tilde X/(H\cap G_1)\ra\tilde X/H$.
\end{proof}}
\begin{rem}\label{rem:pullback}
Assume that  $W=X\# Z$ is irreducible, and let $\tilde f_h:\tilde W\ra Z$ be the Galois closure of $f_h$, and $\Gamma = \Mon_k(f_h)$. 
Then we may identify $\Gamma$ with a (transitive) subgroup of $A$ via the following embedding.
Since $k(W)$ is the compositum of $k(X)$ and $k(Z)$ by Remark \ref{rem:fiber},  the Galois closure $\Omega_{W}$ of $k(W)/k(Z)$ is the compositum of the Galois closure $\Omega_X$ of $k(X)/k(Y)$ with $k(Z)$. Thus $\Gamma = \Gal(\Omega_W/k(Z))$ is isomorphic, via restriction, to $\Gal(\Omega_X/\Omega_X\cap k(Z))\leq A$. 
\end{rem}

\section{Normal and transitive subgroups of imprimitive groups}
\label{sec:toolkit}
Throughout this section, we consider subgroups $G$ of the wreath product $U\wr_J V$, for finite permutation groups $U$ and $V$, with $V$ acting on a set $J$.
\subsection{Normal subgroups}
We start by describing the minimal normal subgroups of $G$.
The following is \DN{essentially} in  \cite{AS}: 
\begin{lem}\label{carmel}
Let $G\leq U\wr_J V$ be a subgroup whose natural projection to $V$ is onto, whose block stabilizer projects onto $U$, and assume $V$ acts transitively on $J$. 
Assume $U$ is primitive with \DN{a unique minimal normal subgroup} $\soc(U)\cong L^I$ for a nonabelian simple group $L$, and 
 $K:=G\cap U^J$ is nontrivial.
\DN{Then $G$ acts transitively on a partition $O_1,\ldots,O_r$ of $I\times J$ such that
$K\cap L^{O_j}\cong L$, $j=1,\ldots,r$ and   $$\soc(K) = K\cap \soc(U)^J \cong (K\cap L^{O_1})\times \cdots \times (K\cap L^{O_r}).$$}
\end{lem}
The proof relies on the following observations and lemma:
\begin{rem}\label{rem:projections}
Suppose $K\leq U^J$ and $V$ is a group of outer automorphisms of $K$ acting transitively by permuting $J$. Then (a) the images of projections $\pi_j:K\ra U$ to the $j$-th coordinate, for $j\in J$, are all isomorphic; and (b) if furthermore $K\neq 1$, $\pi_j(K)\lhd U$, and $U$ has a unique minimal normal subgroup, then  $\pi_j(K)\supseteq \soc(U)$ for all $j\in J$. 

To see (a), let $v_j\in V$ be an automorphism which sends $j$ to $1$, and observe that $\pi_1(K)=\pi_1(v_j(K))\cong \pi_j(K)$ for all   $j\in J$. To see (b), note that since $\soc(U)$ is the unique minimal normal subgroup of $U$ and $\pi_j(K)\lhd U$, the images $\pi_j(K)$, $j\in J$ either contain $\soc(U)$ or are all $\{1\}$. However, the latter does not occur since  $K\neq 1$.
\end{rem}
\DN{The following lemma is a version of the well known Goursat lemma:}
\begin{lem}
\label{lem:gas}\cite[(1.4)]{AS}
Let $L$ be a finite nonabelian simple group, $I$ a finite set, and $K$ a subgroup of $L^I$ which surjects onto $L$ under each projection $\pi_i:K\ra L$ to the $i$-th component for all $i\in I$. 
Then $K$ decomposes as $(K\cap L^{O_1})\times \cdots \times (K\cap L^{O_r})$ 
where 
$O_1,\ldots,O_r$ is a partition of $I$, and $K\cap L^{O_j}\cong L$ for all $j=1,\ldots,r$.
\end{lem}
\begin{proof}[Proof of Lemma \ref{carmel}]
\DN{We first show that the projection of $K\leq U^J$ to the $j$-th component contains  $\soc(U)=L^I$, $j\in J$. 
 First $G$ acts transitively on $J$, so that the assumption of Remark \ref{rem:projections}.a) holds. 
Since the projection $\pi_j:G_0\ra U$ from the $j$-th block stabilizer $G_0$ to the $j$-th copy of $U$ is onto, and since $K\lhd G_0$, we have $\pi_j(K)\lhd U$, $j\in J$.  As in addition $U$ has a unique (nonabelian) minimal normal subgroup, and $K\neq 1$, the conditions of Remark \ref{rem:projections}.(b) also hold. Thus, the remark implies that $\soc(U)=L^I$ is a minimal normal subgroup of $\pi_j(K)\supseteq L^I$, $j\in J$. In particular,  $K_0:=K\cap \soc(U)^J=K\cap L^{I\times J}$ surjects onto each copy of $L$, and hence
$K_0=(K\cap L^{O_1})\times \cdots \times (K\cap L^{O_r})$
with $K\cap  L^{O_k}\cong L$, $k=1,\ldots,r$ for some partition $O_1,\ldots,O_r$ of $I\times J$, by  Lemma \ref{lem:gas}.

Secondly, we show  $G\leq U\wr_J V$ acts transitively on $I\times J$ via conjugation of the $I\times J$-indexed copies of $L$ in $\soc(U)^J$. 
Since $\soc(U)$ is a normal subgroup of the primitive group $U$, it acts transitively on $I$ \cite[Theorem 1.6A]{DM}. As $K_0$ projects onto $\soc(U)$, this implies that $G$ acts transitively on each block $I\times \{j\}$, $j\in J$. 
Furthermore, $G$ acts transitively on the blocks $J$, yielding its transitivity on $I\times J$.
It follows that $G$ acts transitively on the copies of $L$ in $K_0$, and hence that $G$ acts transitively on the partition $O_1,\ldots,O_r$ of $I\times J$. 

It remains to note that $\soc(K)$ in fact equals $K_0$: Since $K_0=K\cap \soc(U)^J$ is normal in $K$ and is a \JKo{direct product of isomorphic nonabelian simple groups which are permuted transitively by $G$,}  $K_0$ is contained in $\soc(K)$. To show equality, it suffices to show that a normal subgroup $C\lhd K$ which is disjoint from $K_0$ is trivial. Indeed, such $C$ centralizes $K_0$, and hence $\pi_j(C)$ centralizers $\soc(U)=L^I\leq \pi_j(K_0)$, $j\in J$. Since the centralizer of $\soc(U)$ in $U$ is trivial, $\pi_j(C)=1$, $j\in J$ and hence $C=1$, as needed.
}
\end{proof}
In particular, in the setting of Lemma \ref{carmel} one has:
\begin{cor}
\label{carmel2}
The socle $\soc(K)$ is a minimal normal subgroup of $G$. 
\end{cor}
\begin{proof} \DN{Let $N\lhd G$ be a normal subgroup}. As in Lemma \ref{lem:gas}, decompose $\soc(K)$ as $\prod_{i=1}^r \soc(K) \cap L^{O_i}$ 
where $O_1,\ldots,O_r$ is a partition of $I\times J$, and $ \soc(K) \cap L^{O_i}\cong L$. 
\DN{Since $N\cap \soc(K)$ is normal in $\soc(K)$,} it decomposes as $\prod_{i\in R}N\cap L^{O_i}$, where $R$ is a subset of $\{1,\ldots,r\}$.
Since $G$ acts transitively on $I\times J$ by Lemma \ref{carmel}, 
 the normality of $N$ in $G$  implies that $R=\{1,\ldots,r\}$ or $\emptyset$, and hence $N\cap \soc(K) = \soc(K)$ or $1$. 
\end{proof}

\subsection{Normal subgroups and decompositions} The following lemma relates normal subgroups of an imprimitive group $G\le U\wr V$ to other partitions of its action. 
\begin{lem}
\label{lem:descent}
Let $G\le U\wr_J V$ be transitive, where $U$ is primitive with \DN{a unique nonabelian minimal normal subgroup}\footnote{In fact $U$ has a unique minimal normal subgroup by Aschbacher--O'Nan--Scott \cite[\S 9]{Gur}.}, and $G$ surjects onto $V$.
Let $G_1\le G$ be a point stabilizer, and $G_1\leq G_0\leq G$  a block stabilizer.
Let $K:=\bigcap_{g\in G}G_0^g$ be the block kernel, and assume $K\neq 1$. 

Then every minimal normal subgroup $N$ of $G$ which is disjoint from $K$ gives rise to 
 a proper subgroup  $G_1 N$ of $G_0N$, with neither of $G_1N$ and $G_0$ containing the other. 
\end{lem}
$$\xymatrix{
G_0N\leq G \ar@{-}[r] \ar@{-}[d] & G_1N \ar@{-}[d] \\
G_0 \ar@{-}[r] & G_1
}
$$
\begin{proof}
To show that $G_1N\ne G_0N$, it suffices to show that $N':= N \cap G_0$ acts trivially on  $G_0/G_1$, since then $(N\cap G_0)G_1 = N'G_1 \ne G_0$, and hence  $G_0\not\leq G_1 N$.

Let $K_0:=\soc(K)$, and let $M$ be the kernel of the action of $K_0\times N'$  on  $G_0/G_1$, so that $(K_0\times N')/M$ embeds into $U$ as a (transitive) normal subgroup. It remains to show that $M$ contains $N'$. However, since  $U$ is nonaffine and $K$ and hence $K_0$ are nontrivial,  $\soc(U)$ and hence also $K_0$ are nontrivial powers of a nonabelian simple group. 
Thus, Goursat's lemma  \cite[Corollary 1.4]{LL} 
implies that a normal subgroup $M$ of $K_0\times N'$ decomposes as $M=(M\cap K_0) \times (M\cap N')$. In particular, the image $K_0/(M \cap K_0) \times N'/(M\cap N')$ of the action is a normal subgroup of $U$.
Since $K_0\neq 1$, it acts nontrivially on each block, and hence  $K_0/(M\cap K_0)$ is nontrivial. As  $U$ has a unique minimal normal subgroup and $K_0\lhd G$, this shows that $K_0/(M\cap K_0)$ contains $\soc(U)$ \DN{as in Remark \ref{rem:projections}}. Moreover, since $U$ has a unique minimal normal subgroup, this forces $N'/(M\cap N')=1$, as desired. 

It remains to note that $G_1N$ is not contained in $G_0$, since by assumption $$1=N\cap K = N\cap \bigcap_{g\in G} G_0^g = \bigcap_{g\in G}(N\cap G_0)^g$$ while $K\neq 1$, giving $N\not\subseteq G_0$. 
\end{proof}

Note that the conclusion of Lemma \ref{lem:descent} yields a refinement $G> G_0N >G_1N>G_1$ of the  inclusion $G>G_1$ which is essentially different from $G>G_0>G_1$ (since neither of $G_1N$ and any conjugate of $G_0$ contain the other).

If $G$ is assumed to be the monodromy group of a polynomial map $f:\mP^1_k\ra \mP^1_k$, then the conclusion gives two essentially different decompositions of $f$, yielding:
\begin{cor}
\label{kor:socGsocK}
Let $k$ be a field of characteristic $0$, and $f_i\in k[x]$, $i=1,\ldots,r$ be indecomposable polynomials with nonsolvable monodromy. Let $A$ be the arithmetic monodromy group of $f=f_1\circ \dots \circ f_r$, and $K$ the kernel of the natural projection $A\to \Mon_k(f_1\circ \cdots \circ f_{r-1})$. Then $\soc(A) = \soc(K)$.
\end{cor}
\begin{rem}\label{rem:min-normal}
Furthermore, we show that $\soc(K)$ is the unique minimal normal subgroup of 
every transitive subgroup \DN{$B\leq A$} containing $\soc(K)$. 
\end{rem}
\begin{proof}[Proof of Corollary \ref{kor:socGsocK} and Remark \ref{rem:min-normal}]
First note that $\soc(K)$ is a minimal normal subgroup of $A$ by Lemma \ref{carmel2}. Ritt's theorem \ref{thm:Ritt} implies that the decomposition of $f$ into indecomposables is unique up to composition with linear polynomials, so that $\soc(K)$ is the unique minimal normal subgroup of $A$ by Lemma \ref{lem:descent}. 
In particular,  $\soc(K)$ has a trivial centralizer in $A$, 
forcing $\soc(B) = \soc(K)$. Finally, since $B$ is transitive, we also deduce that $\soc(K)$ is a minimal normal subgroup of $B$ from Corollary \ref{carmel2} (applied with $U$ replaced by a subgroup of it, namely, the image of the action of a block stabilizer of $B$ on a block). 
\end{proof}

\subsection{Pulling back along covers with \DN{affine monodromy}}
The following lemma is the base of our induction process in Proposition \ref{prop:pols-main} below. 
\begin{lem}\label{lem:affine_sub_transitive}
Let $U\leq G$ be a subgroup such that the action on $G/U$ is affine. Suppose  $G=H_0>H_1>\dots>H_r=:H$ is a chain of maximal subgroups such that 1) the action  $\Gamma_i\leq \Sym(H_{i-1}/H_{i})$ of $H_{i-1}$ on $H_{i-1}/H_{i}$ is almost simple with primitive socle for $i=1,\ldots,r$, and 2) the block kernels $\bigcap_{g\in G} H_i^g$, $i=1,\dots, r$ are pairwise distinct.
Then $H_{i-1}\cap U$ is transitive on $H_{i-1}/H_i$, $i=1,\ldots,r$ and 1) and 2) hold with  $H_i$ replaced by $H_i\cap U$ for $i=0,\dots, r$.
\end{lem}
In terms of coverings this gives the following corollary. As we shall see in Section \ref{sec:poly}, its conditions hold when the coverings are polynomial. 

\begin{cor}
\label{cor:poly_mon_nonaffine}
Let $f_i:X_i\ra X_{i-1}$, $i=1,\ldots,r$ be coverings such that 1) $\Mon(f_i)$ are nonabelian almost simple with primitive socle, and 2) $g_i=f_1\circ\cdots f_i$ does not factor through the Galois closure \DN{$\tilde g_{i-1}$} of $g_{i-1}$ for  $i=1,\ldots,r$.  
Let $h$ be a subcover of $\tilde{g}_r:\tilde{X}\to \mathbb{P}^1_k$ \DN{that has an affine monodromy group}.  Let $g_i'$ be the pullback of $g_i$ along $h$, and define $f_i'$ iteratively via $g_i' = g_{i-1}'\circ f_i'$, $i=1,\ldots, r$. Then $g_i'$ is irreducible, 
 and 1) and 2) above hold with $g_i$ replaced by $g_i'$ and  $f_i$ by $f_i'$.
\end{cor}
\begin{proof}
Pick $G=:H_0>H_1>\cdots >H_r= H$ (resp.\ $U$) so that $f_i$ (resp.\ $h$) is equivalent to the  projection $\tilde X/H_{i-1}\ra \tilde X/H_{i}$, $i=1,\ldots,r$ (resp.\ $\tilde X/U\ra X_0$).  
As the conditions of Lemma \ref{lem:affine_sub_transitive} hold, the lemma implies that  $H_i\cap U$ is transitive on $H_{i-1}/H_i$ so that the pullback  $g_i':\tilde X/(H_i\cap U)\ra\tilde X/U$ of $g_i$ along $h$ is irreducible of the same degree as $g_i$, $i=1,\ldots,r$. 
The second assertion of the corollary now follows directly from the second assertion of Lemma \ref{lem:affine_sub_transitive}.
\end{proof}
 The proof of Lemma \ref{lem:affine_sub_transitive} relies on the following observation. 
\begin{lem}\label{lem:composition-factors}
In the setup of Lemma~\ref{lem:affine_sub_transitive}, 
for every $N\triangleleft G$, 
the group $U\cap N$ must contain all nonabelian composition factors of $N$.\end{lem} 
\begin{proof}
First note that  $U$ contains every nonabelian composition factor of $G$, including multiplicities. 
Indeed, 
since $U$ has an elementary abelian complement, it contains every nonabelian composition factor of $G$. 
Applying this to the quotient  $G/N$ (resp.~$G$) shows that its nonabelian composition factors are the same as those of $U/(N\cap U)$ (resp.~$U$).
 Since the composition factors of $U$ are those of $N\cap U$ combined with those of $U/(N\cap U) \cong UN/N \le G/N$, this implies that the nonabelian composition factors of $N\cap U$ and those of $N$ are the same, as desired.
\end{proof}

\begin{proof}[Proof of Lemma \ref{lem:affine_sub_transitive}]
We first show transitivity by  induction on \DN{$r$}. For the induction base \DN{$r=0$}, the assertion holds trivially. Set $K_i:=\bigcap_{g\in G}H_i^g$, and note that $K_i\neq 1$, $i=1,\ldots,r$.
By induction the action of $UK_{r-1}/K_{r-1}$ on a block $G/H_{r-1}$ is transitive. 
It therefore remains to show that  $U\cap H_{r-1}$ is transitive in its action on a given block $H_{r-1}/H_r$. 
Since $\soc(K_{r-1})\triangleleft G$,  Lemma \ref{lem:composition-factors} shows that $U$ must contain every nonabelian composition factor of $\soc(K_{r-1})$. 
Let $\Gamma$ denote the image of the action $\psi:H_{r-1}\ra\Sym(H_{r-1}/H_r)$. 
Since $\Gamma$ is nonabelian almost simple and $K\ne 1$, Remark \ref{rem:projections} implies that the projection   $\psi(\soc(K_{r-1}))$ to any block is a nontrivial normal subgroup of $\Gamma$. Since $\Gamma$ is primitive, $\psi(\soc(K_{r-1}))$ and 
hence $U\cap H_{r-1}$ is transitive on $H_{r-1}/H_r$, completing the induction.

To show that 1) and 2) hold,  identify the image $\Gamma_i'\leq \Sym((H_{i-1}\cap U)/(H_{i}\cap U))$ of the action of $H_{i-1}\cap U$ with a subgroup of $\Gamma_i$. Then $\Gamma_i'$ is almost simple with the same socle $\soc(\Gamma_i)$, so that 1) holds. 
 Since $\soc(K_i)$ is a direct product of nonabelian simple groups by Lemma \ref{carmel} and its composition factors are contained in $U$ as above, $\soc(K_i)$ is contained in the kernel $\bigcap_{u\in U} (H_{i}\cap U)^u$ of the action of $U$ on cosets of $H_{i}\cap U$, $i=1,\ldots,r$, so that these kernels are pairwise disjoint. 
\end{proof}
\begin{rem}\label{rem:ker-socle}
In the setup of Corollary \ref{cor:poly_mon_nonaffine}, 
we may identify $\Mon_k(g_i')$ as a subgroup of $\Mon_k(g_i)$ as in Remark \ref{rem:pullback}. 
Then the kernel of  $\Mon_k(g_i')\ra \Mon_k(g_{i-1}')$ contains the socle of the kernel of $\Mon_k(g_i)\ra \Mon_k(g_{i-1})$. 
Indeed, as the groups $U$ and $H_i$, $i=1,\ldots,r$  in \DN{the proof of the corollary satisfy all of the hypotheses of Lemma \ref{lem:affine_sub_transitive}, the final step of the proof of the lemma yields the desired conclusion.}
\end{rem}

\section{The main tool}

The following proposition establishes a machinery to compare low genus subcovers of the Galois closure $\tilde f:\tilde X\ra\mP^1_k$ of polynomial coverings $f:\mP^1_k\ra\mP^1_k$, with the composition factors of $f$ itself. 
In this section, we fix a base field $k$ of characteristic $0$. All occurring coverings are to be understood as coverings over $k$. Consequently, the term ``monodromy group" always refers to the arithmetic monodromy group.


\begin{prop}\label{prop:pols-main}
Let $f = f_1\circ \cdots \circ f_r$ for indecomposable polynomials $f_i\in k[x]$ with nonsolvable monodromy group. 
Let $f_V=f_U\circ h'$ be a subcover of the Galois closure of $f$ such that $f_U$ is a composition of coverings with affine monodromy groups while $h'$ is indecomposable whose monodromy group  $\Gamma$ is nonsolvable with a solvable quotient $\Gamma/\soc(\Gamma)$. 
Then there exists a subcover $h$ of $f_V$ with the same Galois closure as $f_1$.  
\end{prop}
\begin{proof}
As usual denote by $f_C:\tilde X/C\ra\mP^1_{k}$ the subcover corresponding to $C\leq \Mon_k(f)$.
First note that the (nonsolvable) monodromy groups of $f_1,\ldots,f_r$ have primitive nonabelian simple socles by Theorem \ref{Burnside}. Moreover,  $g_{i+1}:=f_1\circ \cdots \circ f_{i+1}$ is not a subcover of the Galois closure of $g_i$ by Lemma \ref{lem:pol-kernel}, or equivalently the kernel $K_{i}$ of the projection $\Mon_k(g_i)\ra \Mon_k(g_{i-1})$ is nontrivial, for $i=1,\ldots,r-1$. 

We first show that the pullback $f'$ of $f$ along $f_U$ has similar properties to the above properties of $f$. 
First note that in view of our assumptions on $f_U$ and $f_i$, $i=1,\ldots,r$,  $U$ is transitive by Lemma \ref{lem:affine_sub_transitive}. Thus, letting $G_1\leq \Mon_k(f)$ be a point stabilizer,  $f'$ is equivalent to the projection $\tilde X/(U\cap G_1)\ra \tilde X/U$ by  Remark \ref{rem:fiber}, and $f'$ has the same Galois closure $\tilde X$ as $f$ by Lemma \ref{lem:pull-closure}.    
\DN{As the socles of $\Gamma_i:=\Mon_k(f_i)$ are primitive and nonabelian simple, Corollary \ref{cor:poly_mon_nonaffine} implies that} $f'$ decomposes as $f_1'\circ\cdots \circ f_r'$ for indecomposable $f_i'$, $i=1,\ldots,r$ with almost simple monodromy groups. 
Since $g_i':=f_1'\circ \cdots\circ f_i'$ is the pullback of $g_i$ along $f_U$, we shall henceforth identify $\Mon_k(g_i')$ as a subgroup of $\Mon_k(g_i)$ for all $i$, as in Remark \ref{rem:pullback}. Moreover, Remark \ref{rem:ker-socle} then shows that 
the kernel $K_i'$ of the projection $\Mon_k(g_i')\ra\Mon_k(g_{i-1}')$ contains $\soc(K_i)$ and in particular is nontrivial.

Since $f'$ has Galois closure $\tilde f':\tilde X\ra \tilde X/U$ as shown above, we may regard $h'$ as a subcover of  the Galois closure of $f'$. 
Suppose $1\leq s\leq r$ is minimal for which $h'$ is a subcover of the Galois closure of $g_s'$. We claim that $s=1$. 

Since $\Mon_k(g_i')$ is a transitive subgroup of $\Mon_k(g_i)$ and since $\soc(K_i)\leq \Mon_k(g_i')$ as above, Corollary \ref{kor:socGsocK} and Remark \ref{rem:min-normal} imply that $\soc(K_i)$ is the unique minimal normal subgroup of $\Mon_k(g_i')$. 
Letting $\phi:A\ra \Gamma$ be the natural projection, it follows that either $\soc(K_s)\subseteq \ker\phi$ or $\ker\phi=1$. 
In the former case, since $K_s/\soc(K_s)\leq (\Gamma_s/\soc(\Gamma_s))^{\deg g_{s-1}}$ by Lemma \ref{carmel}, $K_s/\soc(K_s)$ is solvable by \DN{Schreier}'s conjecture, and hence so is  $K'_s/\soc(K_s)$ and $\phi(K'_s)$. 
 Since $\Gamma$ is primitive (as $h'$ is indecomposable) and $\soc(\Gamma)$ is nonabelian, $\Gamma$ has no nontrivial solvable normal subgroups by Aschbacher--O'Nan--Scott \cite[\S 11]{Gur}. Thus $\phi(K'_s)=1$ and $h'$ is a subcover of the Galois closure of $g'_{s-1}$, contradicting the minimality of $s$.

Henceforth, we may assume $\ker\phi=1$, that is, the Galois closure of $h'$ is the same as that of $g_s'$. 
Since  $\Gamma/\soc(\Gamma)$ is solvable, and 
$\Mon_k(g_{s-1}')$ is a nonsolvable quotient of \DN{$\Mon_k(g_s')/\soc(K_s)$} for $s>1$, the claim follows. Note that as $\Mon(f_1')$ is almost simple, the Galois closure of $h'$ in fact coincides with that of $f_1'$.

Let $\tilde f_1:\tilde X_1\ra\mP^1$ be the Galois closure of $f_1$ and $\phi_1:\Mon(f)\ra \Mon(f_1)$ the natural projection. The subcover $h:\tilde X_1/\phi_1(V)\ra \tilde X_1/\Mon(f_1)$ of the Galois closure of $f_1$ is then equivalent to $\tilde X/(\ker(\phi_1)\cdot V)\ra\tilde X/\Mon(f)$, and hence is also a subcover of $f_V:\tilde X/V\ra \tilde X/\Mon(f)$. 
Moreover,  since $\phi_1(V)$ has trivial core 
in $\Mon(f_1')$ and hence in $\Mon(f_1)$, 
the Galois closure of $h$ is the same as that of $f_1$. 
\end{proof}
\DN{Note that the only facts about the polynomials $f_1,\ldots,f_r$ used in the proof are:
 1) $\Mon_k(f_i)$ is almost simple with primitive socle;  2)  $f_1\circ\cdots\circ f_r$ has a unique decomposition up to composition with linear polynomials (Theorem \ref{thm:Ritt}); and 3) $g_{i+1}=f_1\circ\cdots \circ f_{i+1}$ does not factor
 through the Galois closure of $g_i$. Thus, one may replace  $f_1,\ldots,f_r$ in the proposition by arbitrary coverings satisfying 1)-3).}

\section{Proof of Theorems \ref{thm:polsQ} and \ref{thm:DLS-pol}}\label{sec:conclusions}
\subsection{Reductions}\label{sec:reduction}
In this section we apply Proposition \ref{prop:pols-main} to reduce Theorems \ref{thm:polsQ} and \ref{thm:DLS-pol} to assertions regarding merely the first composition factor $f_1$. 
Throughout this section $k$ is finitely generated. For a number field $k$, let $O_k$ be its ring of integers.  
\begin{cor}\label{cor:polsZ}
Suppose $f= f_1\circ \cdots \circ f_r\in k[x]$ for indecomposable polynomials $f_i$, $i=1,\ldots,r$ with nonsolvable monodromy groups over a number field $k$. Then $\Red_f(O_k)$  is contained in the union of a finite set with $\bigcup_h (h(k)\cap O_k)$, where $h$ runs over Siegel functions with the same Galois closure as $f_1$. 
\end{cor}
\begin{proof}
Let $\tilde f:\tilde X\ra\mP^1_{k}$ be the Galois closure of $f$ over $k$, and $A$ its arithmetic monodromy group. 
By Corollary \ref{cor:kronecker}, $\Red_f$ is the union of a finite set with $\bigcup_D f_D(k)$, where $D\leq A$ runs over maximal intransitive subgroups satisfying $D\cdot \Mon_{\oline k}(f) = A$, and $f_D$ is a Siegel function equivalent to the projection $f_D:\tilde X/D\ra\mP^1_k$. 

Since such $D$ is intransitive, Corollary \ref{cor:poly_mon_nonaffine} implies that $f_D$ is not a composition of coverings with affine monodromy. 
Thus, $f_D$ has a minimal subcover $f_{V}:\tilde X/V\ra\mP^1_k$, $D\leq V\leq A$ with decomposition   $f_V  = f_U \circ h'$, $V\leq U\leq A$,  such that $f_U:\tilde X/U\ra\mP^1_k$  is a composition of Siegel functions with affine monodromy and an indecomposable Siegel function $h':\tilde X/V\ra \tilde X/U$ with nonaffine (hence nonsolvable)  monodromy. 

 Proposition \ref{prop:pols-main} then implies that there exists a subcover $h$ of $f_V$ with the same Galois closure as $f_1$. Since $f_D$ is a Siegel function, so is $h$. The claim now follows since clearly $f_D(k)\subseteq h(k)$. 
\end{proof}

\begin{cor}\label{cor:polsQ}
Let $f = f_1\circ \cdots \circ f_r$ for indecomposable  $f_i\in  k[x]$, $i=1,\ldots,r$ 
with nonsolvable monodromy groups, such that the Galois closure of $f_1$ is of genus $>1$. 
Then $\Red_f$ is contained in the union of a finite set and $\bigcup_h h(X(k))$, where $h:X\ra\mP^1_k$ runs over coverings of genus $g_X\leq 1$ with the same Galois closure as $f_1$. 
\end{cor}
\begin{proof}
The proof is similar to that of Corollary \ref{cor:polsZ} but applies Proposition \ref{prop:pols-main} over $\oline k$. 
Letting $\tilde f:\tilde X\ra \mP^1_k$, $\tilde f_{\oline k}:\tilde X_{\oline k} \ra \mP^1_{\oline k}$,  $A$ and $G$, be the Galois closures of $f$ over $k$ and $\oline k$, and arithmetic and geometric mondromy groups of $f$, respectively.   Let $f_D:\tilde X/D\ra\mP^1_k$ be  a genus $\leq 1$ subcover of $\tilde f$ for   maximal intransitive $D\leq A$ satisfying $D\cdot G = A$. 

The  subgroup $C:=D\cap G$  is also intransitive, so that 
the subcover $f_C$ of $\tilde f_{\oline k}$ is not a composition of coverings with affine monodromy. 
Thus, $f_C$ has a minimal subcover $f_{V}:\tilde X_{\oline k}/V\ra\mP^1_{\oline k}$, $C\leq V\leq G$ with decomposition   $f_V  = f_U \circ h'$, $V\leq U\leq G$,  such that $f_U:\tilde X_{\oline k}/U\ra\mP^1_{\oline k}$  is a composition of coverings with affine monodromy, and $h':\tilde X_{\oline k}/V\ra \tilde X_{\oline k}/U$ is an indecomposable covering with nonaffine  (hence nonsolvable) monodromy group. 
Note that $\tilde X_{\oline k}/U$ is of genus $0$, since otherwise $h'$ is a covering between genus $1$ curves, hence with abelian monodromy \cite[Theorem 4.10(c)]{Sil}, contradicting  the assumption that its monodromy is nonaffine. 
Theorem \ref{thm:nonsolv-quot} then  implies that  the proper quotients of $\Mon_{\oline k}(h')$ are solvable.
Thus, the conditions of Proposition \ref{prop:pols-main} hold over  $\oline k$. 

Proposition \ref{prop:pols-main} then implies that there exists a subcover of $f_V$ with the same Galois closure as $f_1$. 
Letting $\pi:A\ra \Gamma_1$ be the projection to $\Gamma_1:=\Mon_k(f_1)$ and noting that $\Gamma_1$ is nonabelian almost simple by Theorem \ref{Burnside}, it follows that $\pi(V)$ and hence $\pi(C)$ does not contain $\soc(\Gamma_1)$. 
We claim that $\pi(D)$ does not contain $\soc(\Gamma_1)$ and hence the natural projection $h:\tilde X/(\ker\pi\cdot D)\ra \mP^1_k$ has the same Galois closure as $f_1$. To see this, assume on the contrary $\pi(D)$ contains $\soc(\Gamma_1)$ and hence $\pi(D)\leq \Gamma_1$ is almost simple with the same socle. 
Note that $C\lhd D$ since $G\lhd A$, and hence $\pi(C)\lhd\pi(D)$.  
 Since $\pi(D)$ is almost simple and $\pi(C)$ does not contain $\soc(\Gamma_1)$, this implies   $\pi(C)=1$. 
 However, the latter 
  contradicts the assumption
  that $f_1$ has Galois closure of genus $>1$, proving the claim.

Since $h$ is a subcover of $f_D$, it follows that $f_D(\tilde X/D(k))\subseteq h(X(k))$, where $X:=\tilde X/(\ker\pi\cdot D)$. Thus Corollary \ref{cor:kronecker} implies that $\Red_f$ is contained in the union of a finite set with $\bigcup_h h(X(k))$, where $h:X\ra\mP^1_k$ runs over coverings of genus $g_X\leq 1$ with the same Galois closure as $f_1$. 
\end{proof}

\begin{cor}\label{cor:DLS-red}
Let $f=f_1\circ\cdots\circ f_r$ be the composition of indecomposable polynomials $f_i\in\mC[x]$ with nonsolvable monodromy group. Assume $f(x)-h(y)\in\mC[x,y]$ is reducible for nonlinear $h\in \mC[y]$. Then $h=h_1\circ g$ for $g,h_1\in \mC[x]$ such that $h_1$ has the same Galois closure as $f_1$. 
\end{cor}
\begin{proof}
The reducibility of $f(x)-h(y)\in \mC[x,y]$ implies the reducibility of the normalization of the curve defined by $f(x)-h(y)=0$.  This curve is (the normalization of) the fiber product of the maps $f:\mP^1_\mC\ra\mP^1_\mC$ and  $h:\mP^1_\mC\ra\mP^1_\mC$. 
By Lemma \ref{Fried}, we may replace $h$ by a common polynomial subcover $h_4$ of $h$ and of the Galois closure $\tilde f$, whose fiber product with $f$ is still reducible. 
It follows that $h_4$ is not a composition of polynomials with affine monodromy groups by Corollary \ref{cor:poly_mon_nonaffine}. 

We may therefore pick a minimal polynomial subcover $h_2\circ h_3$ of $h_4$ which is not a composition of polynomials with affine monodromy groups, so that $h_2$ is a composition of polynomials with affine monodromy groups, and $h_3$ is indecomposable and  $\Gamma:=\Mon_\mC(h_2)$ is nonaffine. As $\Gamma$ is nonaffine, Theorem \ref{Burnside} implies that it is nonabelian almost simple, and in particular $\Gamma/\soc(\Gamma)$ is solvable. We may therefore apply Proposition \ref{prop:pols-main} to deduce that $h_2\circ h_3$ has a polynomial subcover $h_1$ with the same Galois closure as $f_1$. 
 \end{proof}

\subsection{A classification theorem}\label{sec:class} 
The  combination of \cite{Mul} and~\cite[\S1.2]{GS} gives: 
\begin{thm}\label{thm:pol-cls}
Let  $f:\mP^1_{\oline k}\ra\mP^1_{\oline k}$ 
be an indecomposable polynomial covering over $\oline k$ of degree $>20$, and $\tilde f:\tilde X\ra\mP^1$ its Galois closure. For every indecomposable subcover $h:Y\ra\mP^1$ with Galois closure $\tilde f$ and genus  $g_Y\leq 1$, one of the following holds:
\begin{enumerate}
\item $h$ is equivalent to $f$. 
\item $f$ is one of the nine families of polynomials whose ramification is given in Table \ref{table:two-set-stabilizer} with monodromy group $G=A_\ell$ or $S_\ell$; and $h$ is the genus $0$ 
covering $\tilde X/G_2\ra \mP^1$ where $G_2$ is the stabilizer of a set of cardinality $2$. 
\item 
The monodromy group of $f$ is either $\PGammaL_3(4)$ or $\PSL_5(2)$, in their natural action of degree $21$ and $31$, resp. In each case, there is only one possible ramification type for $f$, and exactly one more subcover $h$ of genus $\leq 1$. \footnote{More precisely, $h$ is of genus $0$ and corresponds to the image of the point stabilizer under the graph automorphism. Explicit equations for $f$ and $h$ are given in \cite{CCN99}.}
\end{enumerate}
\end{thm}
\begin{table}
\caption{Ramification types of polynomial maps $\mP^1_{\oline k}\ra\mP^1_{\oline k}$ of degree $\ell> 20$ and monodromy group $A_\ell$ or $S_\ell$
for which the genus of the quotient by a $2$-set stabilizer is $0$.
Here $a\in\{1,\ldots, \ell-1\}$ is odd, $(a,\ell)=1$, and in each type $\ell$ satisfies the necessary congruence conditions to make all exponents integral. } 
\begin{equation*}\label{table:two-set-stabilizer}
\begin{array}{| l |}
\hline
{[\ell], [a,\ell-a], \left[1^{\ell-2}, 2\right]}  \\ 
{[\ell], [1^3,2^{(\ell-3)/2}], [1,2^{(\ell-1)/2}], \left[1^{\ell-2}, 2\right] } \\
{[\ell], [1^2,2^{(\ell-2)/2}] \text{ twice}, \left[1^{\ell-2},2\right]} \\ 
{[\ell], \left[1^3,2^{(\ell-3)/2}\right], [2^{(\ell-3)/2},3] } \\ 
{[\ell], \left[1^2,2^{(\ell-2)/2}\right], [1,2^{(\ell-4)/2},3] } \\ 
{[\ell], \left[1,2^{(\ell-1)/2}\right], [1^2,2^{(\ell-5)/2},3]} \\ 
{[\ell], \left[1^3,2^{(\ell-3)/2}\right], [1,2^{(\ell-5)/2},4] } \\ 
{[\ell], \left[1^2,2^{(\ell-2)/2}\right], [1^2,2^{(\ell-6)/2},4] } \\ 
{[\ell], \left[1,2^{(\ell-1)/2}\right], [1^3,2^{(\ell-7)/2},4]} \\ 
\hline
\end{array}
\end{equation*}
\end{table}

\begin{rem}\label{rem:AG}\label{rem:mon}
1) Note that for polynomials of degree $10\leq \deg f\leq 20$ the corresponding subcovers  $h$   are also listed in \cite{Mul} and \cite[Theorem A.4.1]{GS}. \\
2) In Case (1), either $G=A_n$, $S_n$ or $M_{23}$. 
\DN{Letting  $A=\Mon_k(f)$,  one has $A=G$ in case (3) or if $G=M_{23}$ since in these cases the symmetric normalizer of $G$ is $G$}. \\ 
3) The only groups in Theorem \ref{thm:pol-cls} which appear as the monodromy group of a polynomial  over $\mQ$ are alternating and symmetric. \\
\end{rem}

\subsection{Deducing the theorems}\label{sec:poly}
We first deduce the first assertion of Theorem  \ref{thm:polsQ} from Corollary \ref{cor:polsZ} by applying the classification of Siegel functions:
\begin{proof}[Proof of Theorem \ref{thm:polsQ} for integral values]
Note that since $f_1,\ldots,f_r$ are of degree $\geq 5$ and are not \JKo{linearly related}
 to $x^n$ or Chebyshev,
$\Mon_\mQ(f_i)$, $i=1,\ldots,r$ are nonabelian almost simple by Theorem \ref{Burnside}. Thus,  Corollary \ref{cor:polsZ} implies that  $\Red_{f}(\Z)$ is contained in the union of a finite set with $\bigcup_h (h(\mQ)\cap \Z)$, where $h$ runs over Siegel functions with the same Galois closure as $f_1$.  To show that equality holds, it suffices to show that $h(\mQ)\cap \Z$ is contained in $\Red_{f_1}(\Z)$ for such $h$. 

As $f_1(\mQ)\cap \Z\subseteq \Red_{f_1}(\Z)$, it suffices to consider Siegel functions $h$ arising in a different monodromy action of $\Gamma:=\Mon_\mQ(f_1)$,    
not equivalent to that of $f_1$. 
Since this action may be assumed minimally nonsolvable and $\Gamma$ is almost simple, this means that either $\Gamma$ must induce a Siegel function in a second action {\it permutation-equivalent} to that of $\Mon_\mQ(f_1)$; 
or some subgroup between $\Gamma$ and its socle must induce a Siegel function in a {\it different} primitive action. From the classification of primitive monodromy groups of Siegel functions in \cite{Mul2} (in particular Theorems 4.8 and 4.9), one extracts using a computer check that the first scenario happens only for $\Mon_{\oline k}(f_1)\in \{\PSL_2(11), \PSL_3(2), \PSL_3(3), \PSL_4(2), \PGammaL_3(4), \PSL_5(2)\}$, whereas the second one only happens for $\Mon_{\oline k}(f_1)\in\{A_5, S_5, \PSL_3(2),  \PGammaL_2(9)$, $M_{11}, \PSL_4(2)\}$. Out of those possibilities, only the polynomials with monodromy group $S_5$ and $\PGammaL_2(9)$ can be defined over $\mathbb{Q}$ (Remark \ref{rem:AG}), and for the latter group the Siegel function $h$ does not have two poles of the same order, and so is not a Siegel function over $\mathbb{Q}$, cf.\ e.g.\ \cite[\S 4.4]{Mul2}. 
\end{proof}


Note that by invoking \cite{Mul2} over a general number field $k$ instead, the same proof shows that $\R_f(O_k)$ is the union of $\R_{f_1}(O_k)$ and a finite set if one assumes $k$ is a number field and  merely\footnote{Here $15$ is the degree of the action of $\PSL_4(2)$ from the second list of groups in the above proof. } that $\deg f_1>15$.

The rest of the assertions of Theorem \ref{thm:polsQ} follow from the following theorem which itself follows from Corollary \ref{cor:polsQ} and Theorem \ref{thm:pol-cls}.
\begin{thm}
\label{thm:pols}
Let $k$ be finitely generated, and $f = f_1\circ \cdots \circ f_r$ for indecomposable  $f_i\in k[x]$ of degree $\ge 5$,  none of which is  $\mu\circ x^n\circ\nu$ or $\mu\circ T_n\circ\nu$ for linear $\mu,\nu\in \oline k[x]$. 
  If $\deg(f_1)>20$, then $\R_f$ is the union of $\R_{f_1}$ and a finite set.

In particular,  
either (1) $\R_f$ is the union of $f_1(k)$ and a finite set, or (2) 
there exists (a single)  $f_1'\in k(x)$  such that  $\R_f$ is the union of  $f_1(k)\cup f_1'(k)$ and a finite set.
In the latter case, either the ramification of $f_1$ is as in Table \ref{table:two-set-stabilizer}, or $k\neq \mQ$ and $f_1$ is one of the two cases in Theorem \ref{thm:pol-cls}.(3). 
\end{thm}
\begin{proof}
First note that $\Mon_k(f_i)$, $i=1,\ldots,r$ are nonsolvable by Theorem \ref{Burnside}. Since $\deg f_1>20$ and $\Mon_k(f_1)$ is nonsolvable, the Galois closure of $f_1$ is of genus $>1$, e.g.\ by \cite[Proposition 2.4]{GT}. Thus  Corollary \ref{cor:polsQ} implies that  $\R_{f}$ is contained in the union of a finite set and $\bigcup_h h(X(k))$, where  $h:X\ra\mP^1_k$ runs over genus $\leq 1$ subcovers of the Galois closure of $f_1$.  To show that equality holds, it suffices to show that $h(X(k))$ is contained in $\R_{f_1}$ for all such $h$. Note that  $A:=\Mon_k(f_1)$ is nonabelian almost simple by Theorem \ref{Burnside} since $f_1$ is of degree $\geq 5$ and is not \JKo{linearly related} to $x^n$ or Chebyshev.

By Theorem \ref{thm:pol-cls} and Remark \ref{rem:AG}, the only possible nonsolvable geometric monodromy groups $G=\Mon_{\oline k}(f_1)$ for indecomposable $f_1$ of degree $>20$, which are nonalternating and nonsymmetric,  are $\PGammaL_3(4), M_{23}$ and $\PSL_5(2)$. In these cases $A=G$ and $k\neq \mQ$ by Remark \ref{rem:AG}.(2)-(3). 
Moreover, the Galois closure of the only polynomial covering with monodromy group $M_{23}$ has no other genus $\le 1$ equivalence class of subcovers,  so in this case (1) holds. For $\PGammaL_3(4)$ and $\PSL_5(2)$,  the Galois closures of the corresponding polynomials have only one other equivalence class of subcovers of genus $\le 1$, and its stabilizer is intransitive, whence $h(X(k))\subseteq \R_{f_1}$ and (2) holds.
 Note that these cases do not occur over $\mQ$ as well by 
 Remark~\ref{rem:AG}. 
 
Henceforth, we may assume $A=A_n$ or $S_n$ in their natural action. 
We may assume $h$ is minimal with the same Galois closure as $f_1$, and that it is not equivalent to $f_1$. Let $D\leq A$ be its stabilizer, and set $C:=D\cap G$. 
By \cite[Theorem 5.3]{MN}\footnote{Theorem 5.3 of \cite{MN} is based on the work of Guralnick--Shareshain \cite{GS} but no other results in the classification of monodromy groups are used.}, either $C\geq A_{n-1}$ or $C\gneq A_{n-2}$ and the ramification of $f_1$ is in Table \ref{table:two-set-stabilizer}. As $D\supseteq C$ is maximal for which $A$ acts faithfully on $A/D$, it follows that $D$ is either a stabilizer $A_1:=A\cap S_{n-1}$ in the natural action or a stabilizer $A_2=A\cap (S_{n-2}\times S_2)$ of a set of cardinality $2$, and in the latter case the ramification of $f_1$ is in Table \ref{table:two-set-stabilizer}. The former case $D=A_1$ does not occur since $h$ is not equivalent to $f_1$. In the latter case $D=A_2$, it is intransitive, and hence $h(X(k))\subseteq \R_{f_1}$ and (2) holds. 
\end{proof}

When adding the lists from Remark \ref{rem:AG}.(1) as exceptions to Theorem \ref{thm:pols},  
the same proof would lower the degree assumption on $f_1$ to merely $\deg f_1\geq 10$. In particular over $k=\mQ$,  Theorem \ref{thm:pols} holds for polynomials $f_1$ of degree $\deg f_1> 10$ without adding further exceptions. 

Finally we conclude  Theorem \ref{thm:DLS-pol} from Corollary \ref{cor:DLS-red} and Theorem \ref{thm:pol-cls}:
\begin{proof}[Proof of Theorem \ref{thm:DLS-pol}] 
Since $x^n$, $T_n$, and an indecomposable degree $4$ polynomial do not appear as composition factors of $f$, the monodromy group of each  $f_i$ is nonsolvable with proper solvable quotients by Theorem \ref{Burnside}.
Thus we may apply  Corollary \ref{cor:DLS-red} to obtain a polynomial subcover \DN{$h_1$ of $h$} with the same Galois closure as  of $f_1$.

 Since $\deg f_1>31$, 
  the possibilities for  $h_1$ are described in cases (1)-(2) of Theorem \ref{thm:pol-cls}. 
  In fact, as both $h_1$ and $f_1$ are polynomials with alternating or symmetric monodromy $\Gamma_1$, the point stabilizer of both of them is conjugate to the stabilizer in the natural action of $\Gamma_1$. Hence $h_1$ and $f_1$ are equivalent, as desired. 
  \end{proof}

\appendix

\section{Proof of  Theorem \ref{thm:nonsolv-quot}}\label{app:Shih}
The proof requires the following  addition to Shih's paper \cite{Shih}. More generally, a  classification of indecomposable genus $1$ coverings with more than one minimal normal subgroup is addressed in \cite{Sal}. 
\begin{prop}\label{prop:shih}
Let $f:X\ra\mP^1$ be an indecomposable covering of genus $g_X\leq 1$  whose monodromy group $G$ has more than one minimal normal subgroup. 
Then $G/\soc(G)$ is solvable.  
\end{prop}
\begin{proof}
Since $G$ has more than one minimal normal subgroup, it has two isomorphic nonabelian minimal normal subgroup by the Aschbacher-Scott theorem \cite[Theorem 11.2]{Gur}, so that $\soc(G)\cong L^{2t}$ for a nonabelian simple group $L$ and  $t\geq 1$. 

We claim that the proof in Shih's paper yields that $t=1$ even under the mere assumption $g_X\leq 1$, showing that $\soc(G)\cong L^2$. Since $G/\soc(G)$ embeds into $\Aut(L^2)/\Inn(L)^2\cong (\Aut(L)/\Inn(L))^2\rtimes S_2$ and $\Aut(L)/\Inn(L)$ is solvable by Schreier's conjecture, this shows that $G/\soc(G)$ is solvable, proving the proposition. 

To prove the claim,  
we follow closely the proof of \cite{Shih} and use the notation of \cite[\S 1-2]{Shih}: Let $S=[g_1,\ldots,g_r]$ be a tuple with product $1$ which is  associated to $f$ and generates $G$. Let $\orb(g_1)$ denote the multiset of orbits of $g_1$ and set $n:=\deg f$.  As $g_X\leq 1$, the  Riemann--Hurwitz formula implies that
$$0\geq 2(g_X-1) = -2n + \sum_{i=1}^r(n-\#\orb(x_i)), $$ 
or equivalently $u(S)\geq r-2$, where $u(g_i):=\#\orb(x_i)/n$ and $u(S):=\sum_{i=1}^ru(g_i)$, following the notation of \cite[\S 2]{Shih}. 
Shih's proof uses the assumption $g_X=0$ only in order to make sure the strict inequality $u(S)>r-2$ holds. 
However, we show below that already the inequality $u(S)\geq r-2$ (or  $g_X\leq 1$) suffices for his proof.

As in \cite[(4.8)]{Shih}, one first shows that $r\leq 4$ using \cite[(4.7)]{Shih}. Indeed, by \cite[(4.7).(1)]{Shih}, one has $u(g_i)\leq 3/5$, so that $r-2\leq u(S)\leq 3r/5$. The latter shows that $r\leq 5$ with equality only if $u(g_i)=3/5$ for all $i$. However, \cite[(4.7).(2-3)]{Shih} show that  $u(g_i)\leq \max\{7/20,11/30, 8/15, 11/20\}=11/20<3/5$, so that indeed $r\leq 4$.  

For  $L\neq A_5$, Part (i) of \cite[(4.9)]{Shih} shows that in fact $r=3$. 
For $L=A_5$ and $r=4$, Part (ii) of \cite[(4.9)]{Shih} shows that three of the $g_i$'s are involutions, and the remaining, say $g_4$,  is of order $m\geq 3$. This case is ruled out using merely $u(S)\geq r-2=2$:  Indeed,  $u(g_i)\leq 11/20$ for $i=1,2,3$ by \cite[(4.9).(3)]{Shih} and 
$$u(g_4)\leq 1/m + \frac{m-1}{m}\cdot\frac{1}{n}\max_{1<i<m}\{f(g_4^i)\}$$
by \cite[(2.1).(2)]{Shih}. Since this maximum is at most $1/10$ by \cite[(4.7)]{Shih}, the right hand side is strictly smaller than $1/4+1/10=7/20$ if $m\geq 4$. For $m\geq 4$, in total we have $$u(S)=\sum_{i=1}^3 u(g_i) + u(g_4)< 3\cdot 11/20+7/20=2,$$ contradicting $u(S)\geq 2$. The case $m=3$ is ruled out in \cite[(4.9)]{Shih} already using merely the inequality $u(S)\geq 2$, as needed. 

Henceforth assume $r=3$ and let $k,\ell,m$ be the orders of $g_1,g_2,g_3$, respectively. 
Without loss of generality, we assume $k\leq \ell\leq m$. 
Then  \cite[(4.11)]{Shih} asserts that $(k,\ell,m)$ is one of the tuples $(2,3,m)$, $7\leq m\leq 18$, 
or $(2,4,m)$, $5\leq m\leq 35$, or $(2,5,m)$, $5\leq m\leq 10$, or $(2,6,m)$, $m=6,7,8$, or $(3,3,m)$, $m=4,5$, or $(3,4,4)$. The only estimate of $u(S)$ used are those in \cite[(4.10)]{Shih} whose proof applies in the same way when replacing the inequality $u(S)>1$ with the inequality $u(S)\geq 1$. 

Finally, \cite[(4.17)-(4.21)]{Shih} show that $t=1$:
This relies on \cite[(4.16)]{Shih} which applies the inequality $u(S)\geq 1$ in order to deduce that $(k,\ell,m)$ is $(2,3,8)$, or $(2,4,5)$ or $(2,4,6)$. However, we note that for $(k,\ell,m)=(2,3,7)$ or $(2,3,10)$ the estimates on $u(S)$ in the proof of \cite[(4.16)]{Shih} do not contradict even the inequality $u(S)>1$, leaving these cases open.  In these cases we refine the estimates as follows. Let $f(g)$ denote the number of fixed points of $g\in G$. If $(k,\ell,m)=(2,3,10)$, then  \cite[(2.1).(1)]{Shih} gives: 
\begin{equation}\label{equ:2310} u(g_3)\leq \frac{1}{10}(1+\frac{f(g_3^5)}{n}+4\frac{f(g_3^2)}{n}+4\frac{f(g_3)}{n}).
\end{equation}
One has $f(g_3^5)/n\leq 1/15$  as in \cite[(4.7).(3)]{Shih} for $L\neq A_5$, and $f(g_3^2)/n\leq 1/12$ by \cite[(4.6).(2)]{Shih}, and $f(g_3^2)/n\leq 1/60$ by the assumption of \cite[(4.16)]{Shih}. 
Thus, \eqref{equ:2310} gives $u(g_3)\leq 11/75$. As $u(g_1)+u(g_2)\leq 307/360$ as in \cite[(4.16)]{Shih}, one gets $u(S)\leq 1799/1800$, contradicting $u(S)\geq 1$. 
For $(k,\ell,m)=(2,3,7)$, \cite[(2.1).(1)]{Shih} gives 
$
u(g_3)\leq (1/7)(1+6f(g_3)/n).
$ As $f(g_3)/n\leq 1/60$ by assumption of \cite[(4.16)]{Shih}, one has $u(g_3)\leq 61/420$. As $u(g_1)+u(g_2)\leq 307/360$, this gives $u(S)\leq 503/504$, contradicting $u(S)\geq 1$. 
Now, \cite[(4.16)]{Shih} and the above addition in case $(k,\ell,m)=(2,3,7)$ or $(2,3,10)$ implies that the inequality $u(S)\geq 1$ suffices for \cite[(4.17)-(4.21)]{Shih}.

\end{proof}

\begin{proof}[Proof of Theorem \ref{thm:nonsolv-quot}]
We use the Aschbacher--O'Nan--Scott structure theory of primitive groups \cite[Theorem 11.2]{Gur}. 
Since $G$ is assumed to be nonaffine, $\soc(G)$ is isomorphic to a power $L^t$ of a nonabelian simple group $L$.   
Proposition \ref{prop:shih} shows that either $\soc(G)$ is the (unique) minimal normal subgroup, or $G/\soc(G)$ is solvable. 
Henceforth assume $\soc(G)$ is the unique minimal normal subgroup of $G$.

Consider the action of $G$ on the set $\Delta=\{L_1,\ldots,L_t\}$ of simple direct factors in $\soc(G)$, and let $K$ be its kernel. Since $2(g_X-1)<\deg f/1000$ and $G$ is nonaffine, \cite[Theorem 9.3]{GT} implies that either $G/K$ is solvable or
$G/K\cong A_5$ or $S_5$ with $t=5$.  
Since $L^t\leq K\leq \Aut(L)^t$ and $\Aut(L)/L$ is solvable by Schreier's conjecture, it follows that $G/\soc(G)\cong (G/K)/(K/\soc(G))$ is solvable when $G/K$ is. In particular, since $S_t$ is solvable for $t\leq 4$, we may henceforth assume $t\geq 5$.

\DN{In case $\soc(G)$ acts regularly, we follow the argument of \cite[Corollary 9.4]{GT}:} Since a nontrivial normal subgroup of a primitive group acts transitively, we have  $H\soc(G)=G$, where $H$ is a point stabilizer. Moreover, as $\soc(G)$ is regular,  we have $H\cap \soc(G)=1$ and $G=H\ltimes \soc(G)$. 
Thus $H$ acts transitively on $\Delta$, 
and hence the stabilizer $H_1$ of $L_1\in \Delta$ in this action is of index at least $t\geq 5$, while the kernel $H\cap K$ of \DN{this action is solvable since $H\cap \soc(G)=1$ and $K/\soc(G)$ is solvable. 
However}, the image of the action $H_1\ra\Aut(L_1)$ contains the group of inner automorphisms on $L_1$ by \cite[Theorem 1]{AS}\footnote{More specifically, see the definition of $n_1^*$ in \cite[Theorem 1]{AS}.}. 
\DN{As $H\cap K$ is solvable and is contained in $H_1$, it follows that $H/H\cap K$ contains a nonsolvable subgroup $H_1/H_1\cap K$ of index $\geq 5$. Thus $H/H\cap K\not\cong A_5,S_5$ and is nonsolvable. As $G/K\cong H/H\cap K$ (since $HK=G$) the same conclusion holds for $G/K$, contradicting the above conclusion from \cite[Theorem 9.3]{GT}.}

Henceforth assume $\soc(G)$ is not regular, that is, $H\cap \soc(G)\neq 1$. 
The coverings for which $H\cap L_1= 1$ and $2(g_X-1)<\deg f/168$,  were classified in \cite[Theorem]{Asch}. 
As explained in \cite[(11.1) and (19.1)]{Asch}, the resulting coverings have monodromy group $G\leq S_5\wr S_2$ with  $\soc(G)=A_5^2$, so that $G/\soc(G)$ is solvable. 

Henceforth assume $H\cap L_1\neq 1$, in which case the group is called of product type. In this case,  as $g_X\leq 1$, \cite[Theorem~7.1]{GN}  implies  that\footnote{\cite[Theorem 9.3]{GT} allows a case where $A_5^5<G\leq S_5\wr S_5$, and $G/K\cong S_5$, and the covering $\tilde X/K\ra\mP^1$, obtained as the quotient of the Galois closure $\tilde X$ of $f$ by $K$, has three branch points with branch cycles $x_1,x_2,x_3$ of orders $2,4,6$, respectively. A straightforward computer check shows that there is no  genus $\leq 1$  product $1$ tuple $x_1,x_2,x_3$ generating such $G$, ruling out this case.}  either $G/\soc(G)\cong A_5$ such that its regular action is of genus $0$, or that $G\leq S_\ell \wr S_5$  with $\ell \leq 10$ and $G/\soc(G)\cong S_5$ with regular action of ramification type $[2^{60}], [4^{30}], [5^{24}]$. 
The first case is ruled out by \cite[Theorem 8.6]{GN}, whereas a straightforward computer check shows that the second case does not occur  with genus $g_X\le 1$.
\end{proof}

\bibliographystyle{plain}

\end{document}